\documentclass[10pt,a4paper,twoside]{amsart}

\usepackage{graphicx}
\usepackage{a4wide}

\newtheorem{definition}{Definition}[section]
\newtheorem{theorem}{Theorem}[section]
\newtheorem{lemma}{Lemma}[section]

\newtheorem{remark}{Remark}[section]
\newtheorem*{maintheorem*}{Main Theorem}
\allowdisplaybreaks
\numberwithin{equation}{section}

\newcommand{\norm}[1]{\left\| #1 \right\|}

\newcommand{\eps}{\varepsilon}

\newcommand{\eb}{{\eps,\beta}}
\newcommand{\ueb}{u_\eb}

\newcommand{\pt}{\partial_t}

\newcommand{\px}{\partial_x }
\newcommand{\pxx}{\partial_{xx}^2}
\newcommand{\pxxx}{\partial_{xxx}^3}
\newcommand{\ptxxx}{\partial_{txxx}^4}

\newcommand{\ptxxxx}{\partial_{txxxx}^5}
\newcommand{\ptx}{\partial_{tx}^2}
\newcommand{\ptxx}{\partial_{txx}^3}

\renewcommand{\i}{\ifmmode\mathit{\mathchar"7010 }\else\char"10 \fi}
\renewcommand{\j}{\ifmmode\mathit{\mathchar"7011 }\else\char"11 \fi}
\newcommand{\R}{\mathbb{R}}
\newcommand{\N}{\mathbb{N}}

{%

\begin{enumerate}}%
{\end{enumerate}}

{%

\begin{enumerate}}%
{\end{enumerate}}

\begin{document}\large

\title[A Singular limit problem of Rosenau type]{A singular limit problem for conservation laws \\ related to the Rosenau equation}
\author[G. M. Coclite and L. di Ruvo]{Giuseppe Maria Coclite and Lorenzo di Ruvo}
\address[Giuseppe Maria Coclite and Lorenzo di Ruvo]
{\newline Department of Mathematics,   University of Bari, via E. Orabona 4, 70125 Bari,   Italy}
\email[]{giuseppemaria.coclite@uniba.it, lorenzo.diruvo@uniba.it}
\urladdr{http://www.dm.uniba.it/Members/coclitegm/}




\thanks{The authors are members of the Gruppo Nazionale per l'Analisi Matematica, la Probabilit\`a e le loro Applicazioni (GNAMPA) of the Istituto Nazionale di Alta Matematica (INdAM)}

\begin{abstract}
We consider the Rosenau equation, which contains nonlinear dispersive effects. We prove
that as the diffusion parameter tends to zero, the solutions of the dispersive equation converge to discontinuous weak
solutions of the Burgers equation.
The proof relies on deriving suitable a priori estimates together with an application of the compensated compactness method in the $L^p$ setting.
\end{abstract}

\maketitle


\section{Introduction}\label{sec:intro}
Dynamics of shallow water waves that is observed along lake shores and beaches has been a research area for the past few decades in
oceanography (see \cite{AB,ZZZC}). There are several models proposed in this context: Korteweg-de Vries (KdV) equation, Boussinesq equation, 
Peregrine equation, regularized long wave (RLW) equation, Kawahara equation, Benjamin-Bona-Mahoney equation, Bona-Chen equation etc.
These models were derived from first principles under various different hypothesis and approximations. They are all well studied and very well understood.

The dynamics of dispersive shallow water waves, on the other hand, is captured with slightly different models, like Rosenau-Kawahara equation, 
Rosenau-KdV equation, and Rosenau-KdV-RLW equation \cite{BTL,EMTYB,HXH,LB,RAB}.

The Rosenau-KdV-RLW equation is 
\begin{equation}
\label{eq:RKV-1}
\pt u +a\px u +k\px u^{n}+b_1\pxxx u +b_2\ptxx u + c\ptxxxx u=0,\quad a,\,k,\,b_1,\,b_2,\,c\in\R.
\end{equation}
Here $u(t,x)$ is the nonlinear wave profile. The first term is the linear evolution one, while $a$ is the advection or drifting coefficient. 
$b_1$ and $b_2$ are the dispersion coefficients. 
The higher order dispersion  coefficient is  $c$, 
while the coefficient of nonlinearity is $k$ where $n$ is nonlinearity parameter. 
These are all known and given parameters.

In \cite{RAB}, the authors analyzed \eqref{eq:RKV-1}. They got solitary waves, shock waves and singular solitons along with conservation laws.

Considering the  $n=2,\, a=0,\, k=1,\, b_1=1,\, b_2=-1,\, c=1$:
\begin{equation}
\label{eq:RKV-23}
\pt u +\px u^2 +\pxxx u -\ptxx u +\ptxxxx u=0.
\end{equation}
If $n=2, \, a=0,\, k=1,\, b_1=0,\, b_2=-1,\, c=1$, \eqref{eq:RKV-1} reads
\begin{equation}
\label{eq:RKV-30}
\pt u +\px u^2 -\ptxx u +\ptxxxx u=0,
\end{equation}
which is known as Rosenau-RLW equation.

Arguing in \cite{CdREM}, we re-scale the equations as
follows
\begin{align}
\label{eq:T1}
\pt u +\px u^2 +\beta\pxxx u -\beta\ptxx u +\beta^2\ptxxxx u&=0,\\
\label{eq:T2}
\pt u +\px u^2 -\beta\ptxx u +\beta^2\ptxxxx\ueb&=0,
\end{align}
where $\beta$ is the diffusion parameter.

In \cite{Cd5}, the authors proved that the solutions of \eqref{eq:T1} and \eqref{eq:T2} converge to the unique entropy solution of the  Burgers  equation
\begin{equation}
\label{eq:BU}
\pt u+\px u^2=0.
\end{equation}
Choosing $n=2,\, a=0,\, k=1,\, b_2=b_1=0,\, c=1$, \eqref{eq:RKV-1} reads
\begin{equation}
\label{eq:RKV-2}
\pt u + \px u^{2} + \ptxxxx u=0,
\end{equation}
which is known as Rosenau equation (see \cite{Ro1,Ro2}). 
The existence and the uniqueness of the solution for \eqref{eq:RKV-2} has been proved in \cite{P}.

Finally, if $n=2,\, a=0,\, k=1,\, b_1=1,\, b_2=0,\, c=1$, \eqref{eq:RKV-1} reads
\begin{equation}
\label{eq:RKV-3}
\pt u  +\px u^{2}+\pxxx u + \ptxxxx u=0,
\end{equation}
which is known as Rosenau-KdV equation.

In \cite{Z}, the author discussed the solitary wave solutions and \eqref{eq:RKV-3}. 
In \cite{HXH}, a conservative linear finite difference scheme for the numerical solution for an initial-boundary value problem of the Rosenau-KdV equation 
is considered.
In \cite{E,RTB}, authors discussed the solitary solutions for \eqref{eq:RKV-3} with solitary ansatz method. The authors also gave the two invariants for 
\eqref{eq:RKV-3}. In particular, in \cite{RTB}, the authors  studied  two types of soliton solutions: a solitary wave  and  a singular soliton. 
In \cite{ZZ}, the authors proposed an average linear finite difference scheme for the numerical solution of the initial-boundary value problem for  \eqref{eq:RKV-3}.

In this paper, we analyze  \eqref{eq:RKV-2}. Arguing in \cite{CdREM}, we re-scale the equations as
follows
\begin{equation}
\label{eq:RKV32}
\pt u  + \px u^2 +\beta^2\ptxxxx u=0.
\end{equation}

We are interested in the no high frequency limit,  we send $\beta\to 0$ in \eqref{eq:RKV32}. In this way we pass from \eqref{eq:RKV32}  to \eqref{eq:BU}

We prove that, as $\beta\to0$, the solutions of  converge \eqref{eq:RKV32}  to the unique entropy solution of \eqref{eq:BU}.

In other to do this,  we can choose the initial datum and $\beta$ in two different ways.

Following \cite[Theorem $7.1$]{CRS}, the first choice is the following (see Theorem \ref{th:main-1}):
\begin{equation}
\label{eq:uo-l2}
u_{0}\in L^2(\R), \quad \beta=o\left(\eps^4\right).
\end{equation}

Since $\norm{\cdot}_{L^4}$ is a conserved quantity for \eqref{eq:RKV32},  the second  choice is (see Theorem \ref{th:main-13}):
\begin{equation}
\label{eq:uo-l4}
u_{0}\in L^2(\R)\cap L^4(\R), \quad \beta=\mathbf{\mathcal{O}}\left(\eps^4\right).
\end{equation}
It is interesting to observe that, while the summability on the initial datum in \eqref{eq:uo-l4} is greater than the one in \eqref{eq:uo-l2}, the assumption on $\beta$ in \eqref{eq:uo-l4} is weaker than the one in \eqref{eq:uo-l2}.

From the mathematical point of view, the two assumptions require two different arguments for the $L^{\infty}-$estimate (see Lemmas \ref{lm:50} and \ref{lm:562}). Indeed, the proof of Lemma \ref{lm:50}, under the assumption \eqref{eq:uo-l2}, is more technical than the one of Lemma \ref{lm:562}.

The paper is organized in five sections. In Section \ref{sec:Ro1}, we prove the convergence of \eqref{eq:RKV32} to \eqref{eq:BU} in the $L^p$ setting, with $1\le p<2$. In Section \ref{sec:D1}, we prove the convergence of \eqref{eq:RKV32} to \eqref{eq:BU} in  the $L^p$ setting, with $1\le p<4$. Sections \ref{appen1} and \ref{appen2} are two appendixes, where, choosing the initial datum in two different ways, we prove that the solutions of the  Korteweg-de Vries equation converge to discontinuous weak solutions of \eqref{eq:BU} in the $L^{p}$ setting, with $1\le p< 2$.

\section{The Rosenau equation: $u_0\in L^2(\R)$}\label{sec:Ro1}
In this section, we consider \eqref{eq:RKV32}, and  assume \eqref{eq:uo-l2} on the initial datum.


We study the dispersion-diffusion limit for \eqref{eq:RKV32}, namely we send $\beta\to0$ and get \eqref{eq:BU}. Therefore, we fix two small numbers
$0 < \eps,\,\beta < 1$ and consider the following fifth order problem
\begin{equation}
\label{eq:Ro-eps-beta}
\begin{cases}
\pt\ueb+ \px \ueb^2 +\beta^2\ptxxxx\ueb=\eps\pxx\ueb, &\qquad t>0, \ x\in\R ,\\
\ueb(0,x)=u_{\eps,\beta,0}(x), &\qquad x\in\R,
\end{cases}
\end{equation}
where $u_{\eps,\beta,0}$ is a $C^\infty$ approximation of $u_{0}$ such that
\begin{equation}
\begin{split}
\label{eq:u0eps-1}
&u_{\eps,\,\beta,\,0} \to u_{0} \quad  \textrm{in $L^{p}_{loc}(\R)$, $1\le p < 2$, as $\eps,\,\beta \to 0$,}\\
&\norm{u_{\eps,\beta, 0}}^2_{L^2(\R)}+\left(\beta^{\frac{1}{2}}+ \eps^2\right) \norm{\px u_{\eps,\beta,0}}^2_{L^2(\R)}\le C_0,\quad \eps,\beta >0\\
&\left(\beta^2+\beta\eps^2 \right)\norm{\pxx u_{\eps,\beta,0}}^2_{L^2(\R)} +\beta^{\frac{5}{2}}\norm{\pxxx u_{\eps,\beta,0}}^2_{L^2(\R)}\le C_0,\quad \eps,\beta >0,
\end{split}
\end{equation}
and $C_0$ is a constant independent on $\eps$ and $\beta$.

The main result of this section is the following theorem.
\begin{theorem}
\label{th:main-1}
Assume that \eqref{eq:uo-l2} and  \eqref{eq:u0eps-1} hold. Fix $T>0$,
if
\begin{equation}
\label{eq:beta-eps-2}
\beta=\mathbf{\mathcal{O}}\left(\eps^4\right),
\end{equation}
then, there exist two sequences $\{\eps_{n}\}_{n\in\N}$, $\{\beta_{n}\}_{n\in\N}$, with $\eps_n, \beta_n \to 0$, and a limit function
\begin{equation*}
u\in L^{\infty}((0,T); L^2(\R)),
\end{equation*}
such that
\begin{itemize}
\item[$i)$] $u_{\eps_n, \beta_n}\to u$  strongly in $L^{p}_{loc}(\R^{+}\times\R)$, for each $1\le p <2$,
\item[$ii)$] $u$ is a distributional solution of \eqref{eq:BU}.
\end{itemize}
Moreover, if
\begin{equation}
\label{eq:beta-eps-4}
\beta=o\left(\eps^{4}\right),
\end{equation}
\begin{itemize}
\item[$iii)$] $u$ is the unique entropy solution of \eqref{eq:BU}.
\end{itemize}
\end{theorem}
Let us prove some a priori estimates on $\ueb$, denoting with $C_0$ the constants which depend only on the initial data.

Arguing as \cite[Lemma $2.1$]{Cd5}, we have the following result.
\begin{lemma}\label{lm:38}
For each $t>0$,
\begin{equation}
\label{eq:l-2-u1}
\norm{\ueb(t,\cdot)}^2_{L^2(\R)}+\beta^2\norm{\pxx\ueb(t,\cdot)}^2_{L^2(\R)} + 2\eps\int_{0}^{t}\norm{\px\ueb(s,\cdot)}^2_{L^2(\R)}ds\le C_0.
\end{equation}
\end{lemma}
\begin{lemma}\label{lm:50}
Fix $T>0$. Assume \eqref{eq:beta-eps-2} holds. There exists $C_0>0$, independent on $\eps,\,\beta$ such that
\begin{equation}
\label{eq:u-infty-3}
\norm{\ueb}_{L^{\infty}((0,T)\times\R)}\le C_0\beta^{-\frac{1}{4}}.
\end{equation}
Moreover,
\begin{itemize}
\item[$i)$] the families  $\{\beta^{\frac{1}{2}}\px\ueb\}_{\eps,\,\beta},\,\{\beta^{\frac{1}{4}}\eps\px\ueb\}_{\eps,\,\beta},\,\{\beta^{\frac{3}{4}}\eps\pxx\ueb\}_{\eps,\,\beta}, \{\beta^{\frac{3}{2}}\pxxx\ueb\}_{\eps,\,\beta},$\\ are bounded in $L^{\infty}((0,T);L^{2}(\R))$;
\item[$ii)$] the families $\{\beta^{\frac{3}{4}}\eps^{\frac{1}{2}}\ptx\ueb\}_{\eps,\beta},\,\{\beta^{\frac{7}{4}}\eps^{\frac{1}{2}}\ptxxx\ueb\}_{\eps,\,\beta},\, \{\beta^{\frac{1}{4}}\eps\pt\ueb\}_{\eps,\,\beta} $,\\
$\{\beta^{\frac{5}{4}}\eps^{\frac{1}{2}}\ptxx\ueb\}_{\eps,\,\beta},\, \{\beta^{\frac{1}{2}}\eps^{\frac{1}{2}}\pxx\ueb\}_{\eps,\,\beta}$ are bounded in $L^2((0,T)\times\R)$.
\end{itemize}
\end{lemma}
\begin{proof}
Let $0<t<T$. Multiplying \eqref{eq:Ro-eps-beta} by $-\beta^{\frac{1}{2}}\pxx\ueb -\beta\eps\ptxx\ueb +\eps\pt\ueb$, we have
\begin{equation}
\label{eq:p456}
\begin{split}
&\left(-\beta^{\frac{1}{2}}\pxx\ueb -\beta\eps\ptxx\ueb +\eps\pt\ueb\right)\pt\ueb\\
&\qquad\quad +2\left(-\beta^{\frac{1}{2}}\pxx\ueb -\beta\eps\ptxx\ueb+\eps\pt\ueb\right)\ueb\px\ueb\\
&\qquad\quad +\beta^2\left(-\beta^{\frac{1}{2}}\pxx\ueb -\beta\eps\ptxx\ueb +\eps\pt\ueb\right)\ptxxxx\ueb\\
&\qquad=\eps\left(-\beta^{\frac{1}{2}}\pxx\ueb -\beta\eps\ptxx\ueb +\eps\pt\ueb\right)\pxx\ueb.
\end{split}
\end{equation}
We observe
\begin{equation}
\label{eq:int123}
\begin{split}
&\int_{\R}\left(-\beta^{\frac{1}{2}}\pxx\ueb -\beta\eps\ptxx\ueb+\eps\pt\ueb\right)\pt\ueb dx\\
&\qquad=\frac{\beta^{\frac{1}{2}}}{2}\frac{d}{dt}\norm{\px\ueb(t,\cdot)}^2_{L^2(\R)}+\beta\eps\norm{\ptx\ueb(t,\cdot)}^2_{L^2(\R)}\\
&\qquad\quad +\eps\norm{\pt\ueb(t,\cdot)}^2_{L^2(\R)}.
\end{split}
\end{equation}
Since
\begin{align*}
2\int_{\R}&\left(-\beta^{\frac{1}{2}}\pxx\ueb -\beta\eps\ptxx\ueb+\eps\pt\ueb\right)\ueb\px\ueb\\
=& -2\beta^{\frac{1}{2}}\int_{\R}\ueb\px\ueb\pxx\ueb dx  -2\beta\eps\int_{\R}\ueb\px\ueb\ptxx\ueb dx\\
&+2\eps\int_{\R}\ueb\px\ueb\pt\ueb dx,\\
\beta^2\int_{\R}&\left(-\beta^{\frac{1}{2}}\pxx\ueb -\beta\eps\ptxx\ueb+\eps\pt\ueb\right)\ptxxxx\ueb dx \\
=& \frac{\beta^{\frac{5}{2}}}{2}\frac{d}{dt}\norm{\pxxx\ueb(t,\cdot)}^2_{L^2(\R)}+ \beta^3\eps\norm{\ptxxx\ueb(t,\cdot)}^2_{L^2(\R)}
 + \beta^2\eps\norm{\ptxx\ueb(t,\cdot)}^2_{L^2(\R)},\\
\eps\int_{\R}&\left(-\beta^{\frac{1}{2}}\pxx\ueb -\beta\eps\ptxx\ueb+\eps\pt\ueb\right)\pxx\ueb dx \\
=& -\beta^{\frac{1}{2}}\eps\norm{\pxx\ueb(t,\cdot)}^2_{L^2(\R)}-\frac{\beta\eps^2}{2}\frac{d}{dt}\norm{\pxx\ueb(t,\cdot)}^2_{L^2(\R)}
 -\frac{\eps^2}{2}\frac{d}{dt}\norm{\px\ueb(t,\cdot)}^2_{L^2(\R)}.
\end{align*}
Integrating \eqref{eq:p456} on $\R$ we have
\begin{equation}
\label{eq:501}
\begin{split}
&\frac{d}{dt}\left(\frac{\beta^{\frac{1}{2}}+\eps^2}{2}\norm{\px\ueb(t,\cdot)}^2_{L^2(\R)} 
+\frac{\beta^{\frac{5}{2}}}{2}\norm{\pxxx\ueb(t,\cdot)}^2_{L^2(\R)}\right)\\
&\qquad\quad +\frac{\beta\eps^2}{2}\frac{d}{dt}\norm{\pxx\ueb(t,\cdot)}^2_{L^2(\R)}+\beta\eps\norm{\ptx\ueb(t,\cdot)}^2_{L^2(\R)}\\
&\qquad\quad  + \beta^3\eps\norm{\ptxxx\ueb(t,\cdot)}^2_{L^2(\R)} + \beta^{\frac{1}{2}}\eps\norm{\pxx\ueb(t,\cdot)}^2_{L^2(\R)}\\
&\qquad\quad +\eps\norm{\pt\ueb(t,\cdot)}^2_{L^2(\R)} + \beta^2\eps\norm{\ptxx\ueb(t,\cdot)}^2_{L^2(\R)}\\
&\qquad =2\beta^{\frac{1}{2}}\int_{\R}\ueb\px\ueb\pxx\ueb dx  +2\beta\eps\int_{\R}\ueb\px\ueb\ptxx\ueb dx\\
&\qquad\quad -2\eps\int_{\R}\ueb\px\ueb\pt\ueb dx.
\end{split}
\end{equation}
Due to \eqref{eq:beta-eps-2} and the Young inequality,
\begin{equation}
\label{eq:young24}
\begin{split}
2\beta^{\frac{1}{2}}\int_{\R}&\vert\ueb\vert\vert\px\ueb\vert\vert\pxx\ueb\vert dx
= \beta^{\frac{1}{2}}\int_{\R}\left\vert\frac{2\ueb\px\ueb}{\eps}\right\vert \vert\eps\pxx\ueb\vert dx\\
&\le \frac{2\beta^{\frac{1}{2}}}{\eps}\int_{\R}\ueb^2 (\px\ueb)^2dx + \frac{\beta^{\frac{1}{2}}\eps}{2}\norm{\pxx\ueb(t,\cdot)}^2_{L^2(\R)}\\
& \le C_{0}\eps\norm{\ueb}^2_{L^{\infty}((0,T)\times\R)}\norm{\px\ueb(t,\cdot)}^2_{L^2(\R)} + \frac{\beta^{\frac{1}{2}}\eps}{2}\norm{\pxx\ueb(t,\cdot)}^2_{L^2(\R)},\\
2\beta\eps\int_{\R}&\vert\ueb\vert\vert\px\ueb\vert\vert\ptxx\ueb\vert dx=\eps\int_{\R}\vert 2\ueb\px\ueb\vert \vert\beta\ptxx\ueb\vert dx\\
& \le 2\eps\int_{\R}\ueb^2(\px\ueb)^2 dx + \frac{\beta^2\eps}{2}\norm{\ptxx\ueb(t,\cdot)}^2_{L^2(\R)}\\
& \le 2\eps\norm{\ueb}^2_{L^{\infty}((0,T)\times\R)}\norm{\px\ueb(t,\cdot)}^2_{L^2(\R)} + \frac{\beta^2\eps}{2}\norm{\ptxx\ueb(t,\cdot)}^2_{L^2(\R)},\\
\eps\int_{\R}&\vert 2\ueb\vert\vert\px\ueb\vert\vert\pt\ueb\vert dx \le2\eps\int_{\R}\ueb^2(\px\ueb)^2dx +\frac{\eps}{2}\norm{\pt\ueb(t,\cdot)}^2_{L^2(\R)}\\
& \le 2\eps\norm{\ueb}^2_{L^{\infty}((0,T)\times\R)}\norm{\px\ueb(t,\cdot)}^2_{L^2(\R)}+\frac{\eps}{2}\norm{\pt\ueb(t,\cdot)}^2_{L^2(\R)}.
\end{split}
\end{equation}
From \eqref{eq:501} and \eqref{eq:young24} we gain
\begin{align*}
&\frac{d}{dt}\left(\frac{\beta^{\frac{1}{2}}+\eps^2}{2}\norm{\px\ueb(t,\cdot)}^2_{L^2(\R)} +\frac{\beta^{\frac{5}{2}}}{2}\norm{\pxxx\ueb(t,\cdot)}^2_{L^2(\R)}\right)\\
&\qquad\quad +\frac{\beta\eps^2}{2}\frac{d}{dt}\norm{\pxx\ueb(t,\cdot)}^2_{L^2(\R)}+\frac{\beta\eps}{2}\norm{\ptx\ueb(t,\cdot)}^2_{L^2(\R)}\\
&\qquad\quad  +\beta^3\eps\norm{\ptxxx\ueb(t,\cdot)}^2_{L^2(\R)} + \frac{\beta{\frac{1}{2}}\eps}{2}\norm{\pxx\ueb(t,\cdot)}^2_{L^2(\R)}\\
&\qquad\quad +\frac{\eps}{2}\norm{\pt\ueb(t,\cdot)}^2_{L^2(\R)} + \frac{\beta^2\eps}{2}\norm{\ptxx\ueb(t,\cdot)}^2_{L^2(\R)}\\
&\qquad\le C_{0}\eps \norm{\ueb}^2_{L^{\infty}((0,T)\times\R)}\norm{\px\ueb(t,\cdot)}^2_{L^2(\R)}.
\end{align*}
An integration on $(0,t)$, \eqref{eq:u0eps-1} and \eqref{eq:l-2-u1} give
\begin{equation}
\label{eq:v3256}
\begin{split}
&\frac{\beta^{\frac{1}{2}}+\eps^2}{2}\norm{\px\ueb(t,\cdot)}^2_{L^2(\R)}+\frac{\beta^{\frac{5}{2}}}{2}\norm{\pxxx\ueb(t,\cdot)}^2_{L^2(\R)}\\
&\qquad\quad+\frac{\beta\eps^2}{2}\norm{\pxx\ueb(t,\cdot)}^2_{L^2(\R)}+\frac{\beta\eps}{2}\int_{0}^{t}\norm{\ptx\ueb(s,\cdot)}^2_{L^2(\R)}ds\\
&\qquad\quad+\beta^3\eps\int_{0}^{t}\norm{\ptxxx\ueb(s,\cdot)}^2_{L^2(\R)}ds+\frac{\beta^{\frac{1}{2}}\eps}{2}\int_{0}^{t}\norm{\pxx\ueb(s,\cdot)}^2_{L^2(\R)}ds\\
&\qquad\quad+\frac{\eps}{2}\int_{0}^{t}\norm{\pt\ueb(s,\cdot)}^2_{L^2(\R)}ds+ \frac{\beta^2\eps}{2}\int_{0}^{t}\norm{\ptxx\ueb(t,\cdot)}^2_{L^2(\R)}ds\\
&\qquad\le C_{0} + C_{0}\eps\norm{\ueb}^2_{L^{\infty}((0,T)\times\R)}\int_{0}^{t}\norm{\px\ueb(s,\cdot)}^2_{L^2(\R)}ds\\
&\qquad\le C_{0}\left(1+\norm{\ueb}^2_{L^{\infty}((0,T)\times\R)}\right).
\end{split}
\end{equation}
We prove \eqref{eq:u-infty-3}. Due to \eqref{eq:l-2-u1}, \eqref{eq:v3256} and the H\"older inequality,
\begin{align*}
\ueb^{2}(t,x) =&2\int_{-\infty}^{x}\ueb\px\ueb dx \le 2 \int_{\R}\vert\ueb\vert\vert\px\ueb\vert dx \\
\le& 2\norm{\ueb(t,\cdot)}_{L^2(\R)}\norm{\px\ueb(t,\cdot)}_{L^2(\R)}\\
\le &\frac{C_{0}}{\beta^{\frac{1}{4}}}\sqrt{C_{0} + C_{0}\norm{\ueb}^2_{L^{\infty}((0,T)\times\R)}},
\end{align*}
that is
\begin{equation}
\label{eq:u596}
\norm{\ueb}^4_{L^{\infty}((0,T)\times\R)}\le \frac{C_{0}}{\beta^{\frac{1}{2}}}\left(1+\norm{\ueb}^2_{L^{\infty}((0,T)\times\R)}\right).
\end{equation}
Introducing the notation
\begin{equation}
\label{eq:u678}
y=\norm{\ueb}_{L^{\infty}((0,T)\times\R)},\quad \delta=\beta^{\frac{1}{2}},
\end{equation}
\eqref{eq:u596} reads
\begin{equation*}
y^{4}\le\frac{C_{0}}{\delta}(1+y^2).
\end{equation*}
Arguing as \cite[Lemma $2.3$]{Cd1}, we have
\begin{equation}
\label{eq:sol}
y\le C_{0}\delta^{-\frac{1}{2}}.
\end{equation}
\eqref{eq:u-infty-3} follows from \eqref{eq:u678} and \eqref{eq:sol}.

\eqref{eq:u-infty-3} and \eqref{eq:v3256} give
\begin{align*}
&\frac{\beta^{\frac{1}{2}}+\eps^2}{2}\norm{\px\ueb(t,\cdot)}^2_{L^2(\R)}+\frac{\beta^{\frac{5}{2}}}{2}\norm{\pxxx\ueb(t,\cdot)}^2_{L^2(\R)}\\
&\qquad\quad+\frac{\beta\eps^2}{2}\norm{\pxx\ueb(t,\cdot)}^2_{L^2(\R)}+\frac{\beta\eps}{2}\int_{0}^{t}\norm{\ptx\ueb(s,\cdot)}^2_{L^2(\R)}ds\\
&\qquad\quad+\beta^3\eps\int_{0}^{t}\norm{\ptxxx\ueb(s,\cdot)}^2_{L^2(\R)}ds+\frac{\beta^{\frac{1}{4}}\eps}{2}\int_{0}^{t}\norm{\pxx\ueb(s,\cdot)}^2_{L^2(\R)}ds\\
&\qquad\quad+\frac{\eps}{2}\int_{0}^{t}\norm{\pt\ueb(s,\cdot)}^2_{L^2(\R)}ds+ \frac{\beta^2\eps}{2}\int_{0}^{t}\norm{\ptxx\ueb(t,\cdot)}^2_{L^2(\R)}ds\le C_{0}\beta^{-\frac{1}{2}},
\end{align*}
that is
\begin{align*}
&\frac{\beta+\beta^{\frac{1}{2}}\eps^2}{2}\norm{\px\ueb(t,\cdot)}^2_{L^2(\R)}+\frac{\beta^3}{2}\norm{\pxxx\ueb(t,\cdot)}^2_{L^2(\R)}\\
&\qquad\quad +\frac{\beta^{\frac{3}{2}}\eps^2}{2}\norm{\pxx\ueb(t,\cdot)}^2_{L^2(\R)}+\frac{\beta^{\frac{3}{2}}\eps}{2}\int_{0}^{t}\norm{\ptx\ueb(s,\cdot)}^2_{L^2(\R)}ds\\
&\qquad\quad+\beta^{\frac{7}{2}}\eps\int_{0}^{t}\norm{\ptxxx\ueb(s,\cdot)}^2_{L^2(\R)}ds +\frac{\beta\eps}{2}\int_{0}^{t}\norm{\pxx\ueb(s,\cdot)}^2_{L^2(\R)}ds\\
&\qquad\quad+\frac{\beta^{\frac{1}{2}}\eps}{2}\int_{0}^{t}\norm{\pt\ueb(s,\cdot)}^2_{L^2(\R)}ds + \frac{\beta^{\frac{5}{2}}\eps}{2}\int_{0}^{t}\norm{\ptxx\ueb(t,\cdot)}^2_{L^2(\R)}ds\le C_0.
\end{align*}
Hence,
\begin{align*}
\beta^{\frac{1}{2}}\norm{\px\ueb(t,\cdot)}_{L^2(\R)} \le &C_{0},\\
\beta^{\frac{1}{4}}\eps\norm{\px\ueb(t,\cdot)}_{L^2(\R)} \le &C_{0},\\
\beta^{\frac{3}{2}}\norm{\pxxx\ueb(t,\cdot)}_{L^2(\R)} \le &C_{0},\\
\beta^{\frac{3}{4}}\eps\norm{\pxx\ueb(t,\cdot)}_{L^2(\R)} \le &C_{0},\\
\beta^{\frac{3}{2}}\eps\int_{0}^{t}\norm{\ptx\ueb(s,\cdot)}^2_{L^2(\R)}ds \le &C_{0},\\
\beta^{\frac{7}{2}}\eps\int_{0}^{t}\norm{\ptxxx\ueb(t,\cdot)}^2_{L^2(\R)}ds \le &C_{0},\\
\beta\eps\int_{0}^{t}\norm{\pxx\ueb}^2_{L^2((\R)}ds \le &C_{0},\\
\beta^{\frac{1}{2}}\eps\int_{0}^{t}\norm{\pt\ueb(s,\cdot)}^2_{L^2(\R)}ds \le &C_{0},\\
\beta^{\frac{5}{2}}\eps\int_{0}^{t}\norm{\ptxx\ueb(t,\cdot)}^2_{L^2(\R)} ds\le &C_{0},
\end{align*}
for every $0<t<T$.
\end{proof}
To prove Theorem \ref{th:main-1}. The following technical lemma is needed  \cite{Murat:Hneg}.
\begin{lemma}
\label{lm:1}
Let $\Omega$ be a bounded open subset of $
\R^2$. Suppose that the sequence $\{\mathcal
L_{n}\}_{n\in\mathbb{N}}$ of distributions is bounded in
$W^{-1,\infty}(\Omega)$. Suppose also that
\begin{equation*}
\mathcal L_{n}=\mathcal L_{1,n}+\mathcal L_{2,n},
\end{equation*}
where $\{\mathcal L_{1,n}\}_{n\in\mathbb{N}}$ lies in a
compact subset of $H^{-1}_{loc}(\Omega)$ and
$\{\mathcal L_{2,n}\}_{n\in\mathbb{N}}$ lies in a
bounded subset of $\mathcal{M}_{loc}(\Omega)$. Then $\{\mathcal
L_{n}\}_{n\in\mathbb{N}}$ lies in a compact subset of $H^{-1}_{loc}(\Omega)$.
\end{lemma}
Moreover, we consider the following definition.
\begin{definition}
A pair of functions $(\eta, q)$ is called an  entropy--entropy flux pair if \\
$\eta :\R\to\R$ is a $C^2$ function and $q :\R\to\R$ is defined by
\begin{equation*}
q(u)=2\int_{0}^{u} \xi\eta'(\xi) d\xi.
\end{equation*}
An entropy-entropy flux pair $(\eta,\, q)$ is called  convex/compactly supported if, in addition, $\eta$ is convex/compactly supported.
\end{definition}
We begin by proving the following result.
\begin{lemma}\label{lm:259}
Assume that \eqref{eq:uo-l2}, \eqref{eq:u0eps-1} and \eqref{eq:beta-eps-2} hold. Then for any compactly
supported entropy–-entropy flux pair $(\eta, \,q)$, there exist two sequences $\{\eps_{n}\}_{n\in\N},\,\{\beta_{n}\}_{n\in\N}$, with $\eps_n,\,\beta_n\to0$, and a limit function
\begin{equation*}
u\in L^{\infty}((0,T);L^2(\R)),
\end{equation*}
such that
\begin{align}
\label{eq:con-u-1}
&u_{\eps_{n},\,\beta_{n}}\to u \quad \textrm{in $L^p_{loc}((0,T)\times\R)$, for each $1\le p<2$},\\
\label{eq:u-dist12}
&u \quad \textrm{is a distributional solution of \eqref{eq:BU}}.
\end{align}
\end{lemma}
\begin{proof}
Let us consider a compactly supported entropy--entropy flux pair $(\eta, q)$. Multiplying \eqref{eq:Ro-eps-beta} by $\eta'(\ueb)$, we have
\begin{align*}
\pt\eta(\ueb) + \px q(\ueb) =&\eps \eta'(\ueb) \pxx\ueb +\beta^2\eta'(\ueb)\ptxxxx\ueb \\
=& I_{1,\,\eps,\,\beta}+I_{2,\,\eps,\,\beta}+ I_{3,\,\eps,\,\beta} + I_{4,\,\eps,\,\beta},
\end{align*}
where
\begin{equation}
\begin{split}
\label{eq:12000}
I_{1,\,\eps,\,\beta}&=\px(\eps\eta'(\ueb)\px\ueb),\\
I_{2,\,\eps,\,\beta}&= -\eps\eta''(\ueb)(\px\ueb)^2,\\
I_{3,\,\eps,\,\beta}&= \px(\beta^2\eta'(\ueb)\ptxxx\ueb),\\
I_{4,\,\eps,\,\beta}&= -\beta\eta''(\ueb)\px\ueb\ptxxx\ueb.
\end{split}
\end{equation}
Fix $T>0$. Arguing as \cite[Lemma $3.2$]{Cd2}, we have that $I_{1,\,\eps,\,\beta}\to0$ in $H^{-1}((0,T) \times\R)$, and $\{I_{2,\,\eps,\,\beta}\}_{\eps,\beta >0}$ is bounded in $L^1((0,T)\times\R)$.\\
We claim that
\begin{equation*}
I_{3,\,\eps,\,\beta}\to0 \quad \text{in $H^{-1}((0,T) \times\R),\,T>0,$ as $\eps\to 0$.}
\end{equation*}
By \eqref{eq:beta-eps-2} and Lemma \ref{lm:50},
\begin{align*}
&\norm{ \beta^2\eta'(\ueb)\ptxxx\ueb}^2_{L^2((0,T)\times\R)}\\
&\qquad\le \beta^4 \norm{\eta'}_{L^{\infty}(\R)}\norm{\ptxxx\ueb}^2_{L^2((0,T)\times\R)}\\
&\qquad= \norm{\eta'}_{L^{\infty}(\R)}\frac{\beta^4\eps}{\eps}\norm{\ptxxx\ueb}^2_{L^2((0,T)\times\R)}\\
&\qquad=\norm{\eta'}_{L^{\infty}(\R)}\frac{\beta^{\frac{1}{2}}\beta^{\frac{7}{2}}\eps}{\eps}\norm{\ptxxx\ueb}^2_{L^2((0,T)\times\R)}
\le C_{0}\norm{\eta'}_{L^{\infty}(\R)}\eps\to0.
\end{align*}
We have that
\begin{equation*}
\{I_{4,\,\eps,\,\beta}\}_{\eps,\beta>0} \quad \text{is bounded in $L^1((0,T) \times\R),\,T>0$.}
\end{equation*}
Thanks to \eqref{eq:beta-eps-2}, Lemmas \ref{lm:38}, \ref{lm:50} and the H\"older inequality,
\begin{align*}
&\norm{\beta^2\eta''(\ueb)\px\ueb\ptxxx\ueb}_{L^1((0,T)\times\R)}\\
&\qquad\le\beta^2\norm{\eta''}_{L^{\infty}(\R)}\int_{0}^{T}\!\!\!\int_{\R}\vert\px\ueb\ptxxx\ueb\vert dsdx\\
&\qquad=\norm{\eta''}_{L^{\infty}(\R)}\frac{\beta^2\eps}{\eps}\norm{\px\ueb}_{L^2((0,T)\times\R)}\norm{\ptxxx\ueb}_{L^2((0,T)\times\R)}\\
&\qquad=\norm{\eta''}_{L^{\infty}(\R)}\frac{\beta^{\frac{1}{4}}\beta^{\frac{7}{4}}\eps}{\eps}\norm{\px\ueb}_{L^2((0,T)\times\R)}
\norm{\ptxxx\ueb}_{L^2((0,T)\times\R)}\le C_{0}\norm{\eta''}_{L^{\infty}(\R)}.
\end{align*}
Therefore, \eqref{eq:con-u-1} follows from Lemmas \ref{lm:38}, \ref{lm:1} and the $L^p$ compensated compactness of \cite{SC}.\\
Arguing as \cite[Theorem $2.1$]{Cd5}, we have \eqref{eq:u-dist12}.
\end{proof}
Following \cite{LN}, we prove the following result.
\begin{lemma}\label{lm:452}
Assume that \eqref{eq:uo-l2}, \eqref{eq:u0eps-1} and \eqref{eq:beta-eps-4} hold. Then  for any compactly
supported entropy–-entropy flux pair $(\eta, \,q)$, there exist two sequences $\{\eps_{n}\}_{n\in\N},\,\{\beta_{n}\}_{n\in\N}$, with $\eps_n,\,\beta_n\to0$, and a limit function
\begin{equation*}
u\in L^{\infty}((0,T);L^2(\R)),
\end{equation*}
such that
\eqref{eq:con-u-1} holds and
\begin{equation}
\label{eq:u-entro-sol-12}
u \quad \textrm{is the unique entropy solution of \eqref{eq:BU}}.
\end{equation}
\end{lemma}
\begin{proof}
Let us consider a compactly supported entropy–-entropy flux pair $(\eta,\,q)$. Multiplying \eqref{eq:Ro-eps-beta} by $\eta'(\ueb)$, we have
\begin{align*}
\pt\eta(\ueb) + \px q(\ueb) =&\eps \eta'(\ueb) \pxx\ueb +\beta^2\eta'(\ueb)\ptxxxx\ueb \\
=& I_{1,\,\eps,\,\beta}+I_{2,\,\eps,\,\beta}+ I_{3,\,\eps,\,\beta} + I_{4,\,\eps,\,\beta},
\end{align*}
where $I_{1,\,\eps,\,\beta},\,I_{2,\,\eps,\,\beta},\, I_{3,\,\eps,\,\beta},\, I_{4,\,\eps,\,\beta}$ are defined in \eqref{eq:12000}.

As in Lemma \ref{lm:259}, we obtain that $I_{1,\,\eps,\,\beta}\to 0$ in $H^{-1}((0,T)\times\R)$, $\{I_{2,\,\eps,\,\beta}\}_{\eps,\beta>0}$ is bounded in $L^1((0,T)\times\R)$, $I_{3,\,\eps,\,\beta}\to 0$ in $H^{-1}((0,T)\times\R)$.

Let us show that
\begin{equation*}
I_{4,\,\eps,\,\beta}\to0\quad \text{in $L^1((0,T) \times\R),\,T>0$.}
\end{equation*}
Thanks to \eqref{eq:beta-eps-4}, Lemmas \ref{lm:38}, \ref{lm:50} and the H\"older inequality,
\begin{align*}
&\norm{\beta^2\eta''(\ueb)\px\ueb\ptxxx\ueb}_{L^1((0,T)\times\R)}\\
&\qquad\le\beta^2\norm{\eta''}_{L^{\infty}(\R)}\int_{0}^{T}\!\!\!\int_{\R}\vert\px\ueb\ptxxx\ueb\vert dsdx\\
&\qquad=\norm{\eta''}_{L^{\infty}(\R)}\frac{\beta^2\eps}{\eps}\norm{\px\ueb}_{L^2((0,T)\times\R)}\norm{\ptxxx\ueb}_{L^2((0,T)\times\R)}\\
&\qquad=\norm{\eta''}_{L^{\infty}(\R)}\frac{\beta^{\frac{1}{4}}\beta^{\frac{7}{4}}\eps}{\eps}\norm{\px\ueb}_{L^2((0,T)\times\R)}
\norm{\ptxxx\ueb}_{L^2((0,T)\times\R)}\\
&\qquad\le C_{0}\norm{\eta''}_{L^{\infty}(\R)}\frac{\beta^{\frac{1}{4}}}{\eps}\to 0.
\end{align*}
Arguing as \cite[Theorem $2.1$]{Cd5}, we have \eqref{eq:u-entro-sol-12}.
\end{proof}
\begin{proof}[Proof of Theorem \ref{th:main-1}]
Theorem \ref{th:main-1} follows from Lemmas \ref{lm:259} and \ref{lm:452}.
\end{proof}

\section{The Rosenau equation: $u_0 \in L^2(\R)\cap L^4(\R)$.}\label{sec:D1}
In this section,  we consider \eqref{eq:RKV32}, and  we assume \eqref{eq:uo-l4} on the initial datum.

We consider the approximate problem \eqref{eq:Ro-eps-beta}, where $u_{\eps,\beta,0}$ is a $C^\infty$ approximation of $u_{0}$ such that
\begin{equation}
\begin{split}
\label{eq:u0eps-14}
&u_{\eps,\,\beta,\,0} \to u_{0} \quad  \textrm{in $L^{p}_{loc}(\R)$, $1\le p < 2$, as $\eps,\,\beta \to 0$,}\\
&\norm{u_{\eps,\beta, 0}}^4_{L^4(\R)}+\norm{u_{\eps,\beta, 0}}^2_{L^2(\R)}+(\beta^{\frac{1}{2}}+ \eps^2) \norm{\px u_{\eps,\beta,0}}^2_{L^2(\R)}\le C_0,\quad \eps,\beta >0,\\
&\left(\beta^2+ \beta\eps^2 \right)\norm{\pxx u_{\eps,\beta,0}}^2_{L^2(\R)} +\beta^{\frac{5}{2}}\norm{\pxxx u_{\eps,\beta,0}}^2_{L^2(\R)}\le C_0,\quad \eps,\beta >0,
\end{split}
\end{equation}
and $C_0$ is a constant independent on $\eps$ and $\beta$.

The main result of this section is the following theorem.
\begin{theorem}
\label{th:main-13}
Assume that \eqref{eq:uo-l4} and  \eqref{eq:u0eps-14} hold. Fix $T>0$,
if \eqref{eq:beta-eps-2} holds, there exist two sequences $\{\eps_{n}\}_{n\in\N}$, $\{\beta_{n}\}_{n\in\N}$, with $\eps_n, \beta_n \to 0$, and a limit function
\begin{equation*}
u\in L^{\infty}((0,T); L^2(\R)\cap L^4(\R)),
\end{equation*}
such that
\begin{itemize}
\item[$i)$] $u_{\eps_n, \beta_n}\to u$  strongly in $L^{p}_{loc}(\R^{+}\times\R)$, for each $1\le p <4$,
\item[$ii)$] $u$ is the unique entropy solution of \eqref{eq:BU}.
\end{itemize}
\end{theorem}
Let us prove some a priori estimates on $\ueb$, denoting with $C_0$ the constants which depend only on the initial data.

\begin{lemma}\label{lm:562}
Fix $T>0$. Assume \eqref{eq:beta-eps-2} holds. There exists $C_0>0$, independent on $\eps,\,\beta$ such that \eqref{eq:u-infty-3} holds.
In particular, we have
\begin{equation}
\label{eq:Z45}
\begin{split}
\beta\norm{\px\ueb(t,\cdot)}^2_{L^2(\R)}&+ \beta^3\norm{\pxxx\ueb(t,\cdot)}^2_{L^2(\R)}\\ &+\frac{3\beta\eps}{2}\int_{0}^{t}\norm{\pxx\ueb(s,\cdot)}^2_{L^2(\R)}ds\le C_0,
\end{split}
\end{equation}
for every $0<t<T$. Moreover,
\begin{equation}
\label{eq:Z46}
\norm{\px\ueb}_{L^{\infty}((0,T)\times\R)}\le C_0\beta^{-\frac{3}{4}}.
\end{equation}
\end{lemma}
\begin{remark}
Observe that the proof of Lemma \ref{lm:562} is simpler than the one of Lemma \ref{lm:50}. Indeed, we only need to prove \eqref{eq:u-infty-3}.
\end{remark}
\begin{proof}[Proof of Lemma \ref{lm:562}.]
Let $0<t<T$. Multiplying \eqref{eq:Ro-eps-beta} by $-\beta^{\frac{1}{2}}\pxx\ueb$, we have
\begin{equation}
\label{eq:K23}
-\beta^{\frac{1}{2}}\pxx\ueb\pt\ueb - 2\beta^{\frac{1}{2}}\ueb\px\ueb\pxx\ueb -\beta^{\frac{5}{2}}\ptxxxx\ueb\pxx\ueb=  -\beta^{\frac{1}{2}}\eps(\pxx\ueb)^2.
\end{equation}
Since
\begin{align*}
-\beta^{\frac{1}{2}}\int_{\R}\pxx\ueb\pt\ueb dx = \frac{\beta^{\frac{1}{2}}}{2}\frac{d}{dt}\norm{\px\ueb(t,\cdot)}^2_{L^2(\R)},\\
-\beta^{\frac{5}{2}}\int_{\R}\ptxxxx\ueb\pxx\ueb dx =\frac{\beta^{\frac{5}{2}}}{2}\frac{d}{dt}\norm{\pxxx\ueb(t,\cdot)}^2_{L^2(\R)},
\end{align*}
integrating \eqref{eq:K23} on $\R$, we get
\begin{equation}
\label{eq:K24}
\begin{split}
&\frac{d}{dt}\left(\beta^{\frac{1}{2}} \norm{\px\ueb(t,\cdot)}^2_{L^2(\R)} + \beta^{\frac{5}{2}}\norm{\pxxx\ueb(t,\cdot)}^2_{L^2(\R)}\right)\\
&\qquad\quad +2\beta^{\frac{1}{2}}\eps\norm{\pxx\ueb(t,\cdot)}^2_{L^2(\R)}= 4\beta^{\frac{1}{2}}\int_{\R}\ueb\px\ueb\pxx\ueb dx.
\end{split}
\end{equation}
It follows from \eqref{eq:young24} and \eqref{eq:K24} that
\begin{align*}
&\frac{d}{dt}\left(\beta^{\frac{1}{2}} \norm{\px\ueb(t,\cdot)}^2_{L^2(\R)} + \beta^{\frac{5}{2}}\norm{\pxxx\ueb(t,\cdot)}^2_{L^2(\R)}\right)\\
&\qquad\quad +\frac{3}{2}\beta^{\frac{1}{2}}\eps\norm{\pxx\ueb(t,\cdot)}^2_{L^2(\R)}\le C_0\eps\norm{\ueb}^2_{L^{\infty}((0,T)\times\R)}\norm{\px\ueb(t,\cdot)}^2_{L^2(\R)}.
\end{align*}
Integrating on $(0,t)$, from \eqref{eq:l-2-u1} and \eqref{eq:u0eps-14}, we get
\begin{equation}
\label{eq:K26}
\begin{split}
\beta^{\frac{1}{2}} \norm{\px\ueb(t,\cdot)}^2_{L^2(\R)} &+ \beta^{\frac{5}{2}}\norm{\pxxx\ueb(t,\cdot)}^2_{L^2(\R)}\\
&+\frac{3}{2}\beta^{\frac{1}{2}}\eps\int_{\R}\norm{\pxx\ueb(s,\cdot)}^2_{L^2(\R)}ds\\
\le& C_{0} + C_{0}\eps\norm{\ueb}^2_{L^{\infty}((0,T)\times\R)}\int_{0}^{t}\norm{\px\ueb(s,\cdot)}^2_{L^2(\R)}ds\\
\le& C_{0}\left(1+\norm{\ueb}^2_{L^{\infty}((0,T)\times\R)}\right).
\end{split}
\end{equation}
We prove \eqref{eq:u-infty-3}. Due to \eqref{eq:l-2-u1}, \eqref{eq:K26} and the H\"older inequality,
\begin{align*}
\ueb^2(t,x) = &2 \int_{-\infty}^{x}\ueb\px\ueb dx \le \int_{\R} \vert\ueb\px\ueb\vert dx \\
&\le \norm{\ueb(t,\cdot)}^2_{\R}\norm{\px\ueb(t,\cdot)}^2_{L(\R)}\\
&\le \frac{C_0}{\beta^{\frac{1}{4}}}\sqrt{\left(1+\norm{\ueb}^2_{L^{\infty}((0,T)\times\R)}\right)},
\end{align*}
that is
\begin{equation*}
\norm{\ueb}^4_{L^{\infty}((0,T)\times\R)}\le \frac{C_0}{\beta^{\frac{1}{2}}}\left(1+\norm{\ueb}^2_{L^{\infty}((0,T)\times\R)}\right).
\end{equation*}
Arguing as Lemma \ref{lm:50}, we have \eqref{eq:u-infty-3}

\eqref{eq:Z45} follows from \eqref{eq:u-infty-3} and \eqref{eq:K26}.

Finally, we prove \eqref{eq:Z46}. Due to \eqref{eq:l-2-u1}, \eqref{eq:Z45} and the H\"older inequality,
\begin{align*}
(\px\ueb(t,x))^2 =&2\int_{-\infty}^{x}\px\ueb\pxx\ueb dy \le 2\int_{\R} \px\ueb\pxx\ueb dx \\
\le &\norm{\px\ueb(t,\cdot)}^2_{L^2(\R)}\norm{\pxx\ueb(t,\cdot)}^2_{L^2(\R)}\le C_{0}\beta^{-\frac{3}{2}}.
\end{align*}
Hence,
\begin{equation*}
\vert\px\ueb\vert \le C_0\beta^{-\frac{3}{4}},
\end{equation*}
which gives \eqref{eq:Z46}.
\end{proof}
Following \cite[Lemma $2.2$]{Cd}, or \cite[Lemma $4.2$]{CK}, we prove the following result.
\begin{lemma}\label{lm:n2}
Fix $T>0$. Assume \eqref{eq:beta-eps-2} holds. Then:
\begin{itemize}
\item[$i)$] the family  $\{\ueb\}_{\eps,\,\beta}$ is bounded in $L^{\infty}((0,T);L^{4}(\R))$;
\item[$ii)$] the families $\{\eps\px\ueb\}_{\eps,\,\beta},\,\{\beta^{\frac{1}{2}}\eps^{\frac{1}{2}}\pxx\ueb\}_{\eps,\,\beta}$ are bounded in $L^{\infty}((0,T);L^{2}(\R))$;
\item[$iii)$] the families $\{\beta^{\frac{1}{2}}\eps^{\frac{1}{2}}\ptx\ueb\}_{\eps,\,\beta},\,\{\eps^{\frac{1}{2}}\pt\ueb\}_{\eps,\,\beta},\, \{\beta^{\frac{3}{2}}\eps^{\frac{1}{2}}\ptxxx\ueb\}_{\eps,\,\beta},$\\
$\{\beta\eps^{\frac{1}{2}}\ptxx\ueb\}_{\eps,\,\beta},\,\{\eps^{\frac{1}{2}}\ueb\px\ueb\}_{\eps,\,\beta}$ are bounded in $L^{2}((0,T)\times\R)$.
\end{itemize}
\end{lemma}
\begin{proof}
Let $0<t<T$. Let $A,\,B$ be some positive constants which will be specified later. Multiplying \eqref{eq:Ro-eps-beta} by
\begin{equation*}
\ueb^3 -A\beta\eps\ptxx\ueb -B\eps\pt\ueb,
\end{equation*}
we have
\begin{equation}
\label{eq:P1}
\begin{split}
&\left(\ueb^3 -A\beta\eps\ptxx\ueb +B\eps\pt\ueb\right)\pt\ueb\\
&\qquad\quad+2\left(\ueb^3 -A\beta\eps\ptxx\ueb +B\eps\pt\ueb\right)\ueb\px\ueb\\
&\qquad\quad +\beta^2 \left(\ueb^3 -A\beta\eps\ptxx\ueb +B\eps\pt\ueb\right)\ptxxxx\ueb\\
&\qquad= \eps\left(\ueb^3 -A\beta\eps\ptxx\ueb +B\eps\pt\ueb\right)\pxx\ueb.
\end{split}
\end{equation}
We observe that
\begin{align*}
\int_{\R}&\left(\ueb^3 -A\beta\eps\ptxx\ueb +B\eps\pt\ueb\right)\pt\ueb dx\\
=&\frac{1}{4}\norm{\ueb(t,\cdot)}^4_{L^4(\R)} + A\beta\eps\norm{\ptx\ueb(t,\cdot)}^2_{L^2(\R)} +B\eps\norm{\pt\ueb(t,\cdot)}^2_{L^2(\R)},\\
2\int_{\R}&\left(\ueb^3 -A\beta\eps\ptxx\ueb +B\eps\pt\ueb\right)\ueb\px\ueb dx \\
=&  -2A\beta\eps\int_{\R}\ueb\px\ueb\ptxx\ueb dx +2B\eps\int_{\R}\ueb\px\ueb\pt\ueb dx,\\
\beta^2\int_{\R}&\left(\ueb^3 -A\beta\eps\ptxx\ueb +B\eps\pt\ueb\right)\ptxxxx\ueb dx\\
=& -3\beta^2\int_{\R} \ueb^2 \px\ueb \ptxxx\ueb dx +A\beta^3\eps\norm{\ptxxx\ueb(t,\cdot)}^2_{L^2(\R)}\\
& +B\beta^2\eps\norm{\ptxx\ueb(t,\cdot)}^2_{L^2(\R)},\\
\eps\int_{\R}&\left(\ueb^3 -A\beta\eps\ptxx\ueb +B\eps\pt\ueb\right)\pxx\ueb dx \\
=& -3\eps\norm{\ueb(t,\cdot)\px\ueb(t,\cdot)}^2_{L^2(\R)} -\frac{A\beta\eps^2}{2}\frac{d}{dt}\norm{\pxx\ueb(t,\cdot)}^2_{L^2(\R)}\\
& -\frac{B\eps^2}{2}\frac{d}{dt}\norm{\px\ueb(t,\cdot)}^2_{L^2(\R)}.
\end{align*}
An integration of \eqref{eq:P1} on $\R$ gives
\begin{equation}
\label{eq:P6}
\begin{split}
&\frac{d}{dt}\left(\frac{1}{4}\norm{\ueb(t,\cdot)}^4_{L^4(\R)}+ \frac{B\eps^2}{2}\norm{\px\ueb(t,\cdot)}^2_{L^2(\R)} +  \frac{A\beta\eps}{2}\norm{\pxx\ueb(t,\cdot)}^2_{L^2(\R)}\right)\\
&\qquad\quad + A\beta\eps\norm{\ptx\ueb(t,\cdot)}^2_{L^2(\R)}+B\eps\norm{\pt\ueb(t,\cdot)}^2_{L^2(\R)}\\
&\qquad\quad +A\beta^3\eps\norm{\ptxxx\ueb(t,\cdot)}^2_{L^2(\R)} +B\beta^2\eps\norm{\ptxx\ueb(t,\cdot)}^2_{L^2(\R)}\\
&\qquad\quad +3\eps\norm{\ueb(t,\cdot)\px\ueb(t,\cdot)}^2_{L^2(\R)}\\
&\qquad= 2A\beta\eps\int_{\R}\ueb\px\ueb\ptxx\ueb dx -2B\eps\int_{\R}\ueb\px\ueb\pt\ueb dx\\
&\qquad\quad +3\beta^2\int_{\R} \ueb^2 \px\ueb \ptxxx\ueb dx.
\end{split}
\end{equation}
Due to the Young inequality,
\begin{align*}
2A\beta\eps\int_{\R}&\vert\ueb\px\ueb\vert\vert\ptxx\ueb\vert dx = \eps\int_{\R}\left\vert\frac{4A\ueb\px\ueb}{\sqrt{B}}\right\vert\left\vert\sqrt{B}\beta
\ptxx\ueb\right\vert dx\\
\le& \frac{8A^2\eps}{B}\norm{\ueb(t,\cdot)\px\ueb(t,\cdot)}^2_{L^2(\R)} +\frac{B\beta^2\eps}{2}\norm{\ptxx\ueb(t,\cdot)}^2_{L^2(\R)},\\
2B\eps\int_{\R}&\vert\ueb\px\ueb\vert\vert\pt\ueb\vert dx= \eps\int_{\R}\left \vert \ueb\px\ueb\right\vert \left \vert 2B \pt\ueb\right\vert dx \\
\le& \frac{\eps}{2}\norm{\ueb(t,\cdot)\px\ueb(t,\cdot)}^2_{L^2(\R)} + 2B^2\norm{\pt\ueb(t,\cdot)}^2_{L^2(\R)}.
\end{align*}
Therefore, from \eqref{eq:P6} we gain
\begin{equation}
\label{eq:P8}
\begin{split}
&\frac{d}{dt}\left(\frac{1}{4}\norm{\ueb(t,\cdot)}^4_{L^4(\R)}+ \frac{B\eps}{2}\norm{\px\ueb(t,\cdot)}^2_{L^2(\R)} +  \frac{A\beta\eps}{2}\norm{\pxx\ueb(t,\cdot)}^2_{L^2(\R)}\right)\\
&\qquad\quad + A\beta\eps\norm{\ptx\ueb(t,\cdot)}^2_{L^2(\R)}+\eps B\left (1 - 2B\right)\norm{\pt\ueb(t,\cdot)}^2_{L^2(\R)}\\
&\qquad\quad +A\beta^3\eps\norm{\ptxxx\ueb(t,\cdot)}^2_{L^2(\R)} +\frac{B\beta^2\eps}{2}\norm{\ptxx\ueb(t,\cdot)}^2_{L^2(\R)}\\
&\qquad\quad +\eps\left(\frac{5}{2} -\frac{8A^2}{B}\right)\norm{\ueb(t,\cdot)\px\ueb(t,\cdot)}^2_{L^2(\R)}\\
&\qquad\le 3\beta^2\int_{\R} \ueb^2 \vert\px\ueb\vert\vert \ptxxx\ueb \vert dx.
\end{split}
\end{equation}
From \eqref{eq:beta-eps-2}, we have
\begin{equation}
\label{eq:P9}
\beta \le D^2 \eps^4,
\end{equation}
where $D$ is a positive constant which will be specified later. Due to \eqref{eq:u-infty-3}, \eqref{eq:P9} and the Young inequality,
\begin{align*}
&3\beta^2\int_{\R} \ueb^2 \vert\px\ueb\vert\vert \ptxxx\ueb \vert dx = \int_{\R}\left \vert  \frac{3\beta^{\frac{1}{2}} \ueb^2\px\ueb}{\eps^{\frac{1}{2}}\sqrt{A}}\right\vert \left\vert \beta^{\frac{3}{2}}\eps^{\frac{1}{2}}\sqrt{A}\ptxxx\ueb\right\vert dx\\
&\qquad \le \frac{9\beta}{2\eps A}\int_{\R}\ueb^4(\px\ueb)^2 dx + \frac{A\beta^3\eps }{2}\norm{\ptxxx\ueb(t,\cdot)}^2_{L^2(\R)}\\
&\qquad \le \frac{9\beta}{2\eps A}\norm{\ueb}^2_{L^{\infty}((0,T)\times\R)}\norm{\ueb(t,\cdot)\px\ueb(t,\cdot)}^2_{L^2(\R)}\\
&\qquad\quad+  \frac{A\beta^3\eps }{2}\norm{\ptxxx\ueb(t,\cdot)}^2_{L^2(\R)}\\
&\qquad\le \frac{C_{0}D\eps}{A}\norm{\ueb(t,\cdot)\px\ueb(t,\cdot)}^2_{L^2(\R)} + \frac{A\beta^3\eps }{2}\norm{\ptxxx\ueb(t,\cdot)}^2_{L^2(\R)}.
\end{align*}
Then, \eqref{eq:P8} gives
\begin{equation}
\label{eq:P11}
\begin{split}
&\frac{d}{dt}\left(\frac{1}{4}\norm{\ueb(t,\cdot)}^4_{L^4(\R)}+ \frac{B\eps^2}{2}\norm{\px\ueb(t,\cdot)}^2_{L^2(\R)} +  \frac{A\beta\eps}{2}\norm{\pxx\ueb(t,\cdot)}^2_{L^2(\R)}\right)\\
&\qquad\quad + A\beta\eps\norm{\ptx\ueb(t,\cdot)}^2_{L^2(\R)}+\eps B\left (1 - 2B\right)\norm{\pt\ueb(t,\cdot)}^2_{L^2(\R)}\\
&\qquad\quad +\frac{A\beta^3\eps}{2}\norm{\ptxxx\ueb(t,\cdot)}^2_{L^2(\R)} +\frac{B\beta^2\eps}{2}\norm{\ptxx\ueb(t,\cdot)}^2_{L^2(\R)}\\
&\qquad\quad +\eps\left(\frac{5}{2} -\frac{8A^2}{B}- \frac{C_{0}D}{A}\right)\norm{\ueb(t,\cdot)\px\ueb(t,\cdot)}^2_{L^2(\R)}\le 0.
\end{split}
\end{equation}
We search $A,\,B$ such that
\begin{equation*}
\begin{cases}
\displaystyle 1-2B >0,\\
\displaystyle \frac{5}{2} -\frac{8A^2}{B}- \frac{C_{0}D}{A} >0,\\
\end{cases}
\end{equation*}
that is
\begin{equation}
\label{eq:p302}
\begin{cases}
\displaystyle B<\frac{1}{2},\\
\displaystyle 16A^3 -5BA +2C_0 BD <0.
\end{cases}
\end{equation}
We choose
\begin{equation}
\label{eq:scet-B}
B=\frac{1}{3}.
\end{equation}
Therefore, the second equation of \eqref{eq:p302} reads
\begin{equation*}
16A^3 -\frac{5}{3}A +\frac{2}{3}C_0D <0,
\end{equation*}
that is
\begin{equation}
\label{eq:P12}
48A^3 - 5A +2C_{0}D <0.
\end{equation}
Let us consider the following function
\begin{equation}
\label{eq:Def-g}
g(X) = 48 X^3 - 5X +2C_{0}D
\end{equation}
We observe that
\begin{equation}
\label{eq:P13}
\lim_{x\to -\infty} g(X) =-\infty, \quad g(0)= 2C_0D>0, \quad \lim_{x\to\infty}g(X)=\infty.
\end{equation}
Since $g'(X)= 144 X^2 -5 $, we find that
\begin{equation}
\label{eq:zu1}
g\quad \textrm{is increasing in}\> \left(-\infty, -\frac{\sqrt{5}}{12}\right)\>\text{and in}\>\left(\frac{\sqrt{5}}{12}, \infty\right).
\end{equation}
Therefore,
\begin{equation}
\label{eq:P14}
g\left(\frac{\sqrt{5}}{12}\right)= 16\left(\frac{\sqrt{5}}{12}\right)^3  - \frac{5\sqrt{5}}{12} + 2C_{0}D.
\end{equation}
Since we want that
\begin{equation}
\label{eq:P15}
g\left(\frac{\sqrt{5}}{12}\right)<0,
\end{equation}
we choose
\begin{equation}
\label{eq:P16}
D< \frac{5\sqrt{5}}{27C_{0}}.
\end{equation}
It follows from \eqref{eq:P13}, \eqref{eq:zu1}, \eqref{eq:P14}, \eqref{eq:P15}, and \eqref{eq:P16} that the function $g$ has three zeros $A_1< 0<A_2<A_3$.

Therefore, \eqref{eq:P12} is verified when
\begin{equation}
\label{eq:sce-A}
A_2<A<A_3.
\end{equation}
From \eqref{eq:P11}, \eqref{eq:scet-B}, and \eqref{eq:sce-A}, we have
\begin{align*}
&\frac{d}{dt}\left(\frac{1}{4}\norm{\ueb(t,\cdot)}^4_{L^4(\R)}+ \frac{\eps^2}{6}\norm{\px\ueb(t,\cdot)}^2_{L^2(\R)} +  \frac{A\beta\eps}{2}\norm{\pxx\ueb(t,\cdot)}^2_{L^2(\R)}\right)\\
&\qquad\quad + A\beta\eps\norm{\ptx\ueb(t,\cdot)}^2_{L^2(\R)}+\frac{\eps}{9}\norm{\pt\ueb(t,\cdot)}^2_{L^2(\R)}\\
&\qquad\quad +\frac{A\beta^3\eps}{2}\norm{\ptxxx\ueb(t,\cdot)}^2_{L^2(\R)} +\frac{\beta^2\eps}{6}\norm{\ptxx\ueb(t,\cdot)}^2_{L^2(\R)}\\
&\qquad\quad +K_1\eps\norm{\ueb(t,\cdot)\px\ueb(t,\cdot)}^2_{L^2(\R)}\le 0,
\end{align*}
where $K_1$ is a positive constant.

An integration on $(0,t)$ and \eqref{eq:u0eps-14} give
\begin{align*}
&\frac{1}{4}\norm{\ueb(t,\cdot)}^4_{L^4(\R)}+ \frac{\eps^2}{6}\norm{\px\ueb(t,\cdot)}^2_{L^2(\R)} +  \frac{A\beta\eps}{2}\norm{\pxx\ueb(t,\cdot)}^2_{L^2(\R)}\\
&\qquad\quad + A\beta\eps\int_{0}^{t}\norm{\ptx\ueb(s,\cdot)}^2_{L^2(\R)}ds+\frac{\eps}{9}\int_{0}^{t}\norm{\pt\ueb(s,\cdot)}^2_{L^2(\R)}ds\\
&\qquad\quad +\frac{A\beta^3\eps}{2}\int_{0}^{t}\norm{\ptxxx\ueb(s,\cdot)}^2_{L^2(\R)}ds +\frac{\beta^2\eps}{6}\int_{0}^{t}\norm{\ptxx\ueb(s,\cdot)}^2_{L^2(\R)}ds\\
&\qquad\quad +K_1\eps\int_{0}^{t}\norm{\ueb(s,\cdot)\px\ueb(t,\cdot)}^2_{L^2(\R)}ds\le C_0.
\end{align*}
Hence,
\begin{align*}
\norm{\ueb(t,\cdot)}_{L^4(\R)}\le & C_{0},\\
\eps\norm{\px\ueb(t,\cdot)}_{L^2(\R)}\le &C_{0},\\
\beta^{\frac{1}{2}}\eps^{\frac{1}{2}}\norm{\pxx\ueb(t,\cdot)}_{L^2(\R)}\le &C_{0},\\
\beta\eps\int_{0}^{t}\norm{\ptx\ueb(s,\cdot)}^2_{L^2(\R)}ds\le &C_{0},\\
\eps\int_{0}^{t}\norm{\pt\ueb(s,\cdot)}^2_{L^2(\R)}ds\le &C_{0},\\
\beta^3\eps\int_{0}^{t}\norm{\ptxxx\ueb(s,\cdot)}^2_{L^2(\R)}ds\le& C_{0},\\
\beta^2\eps\int_{0}^{t}\norm{\ptxx\ueb(s,\cdot)}^2_{L^2(\R)}ds\le& C_{0},\\
\eps\int_{0}^{t}\norm{\ueb(s,\cdot)\px\ueb(t,\cdot)}^2_{L^2(\R)}\le& C_{0},
\end{align*}
for every $0<t<T$.
\end{proof}
We are ready for the proof of Theorem \ref{th:main-13}.
\begin{proof}[Proof of Theorem \eqref{th:main-13}.]
Let us consider a compactly supported entropy--entropy flux pair $(\eta,\,q)$. Multiplying \eqref{eq:Ro-eps-beta} by $\eta'(\ueb)$, we have
\begin{align*}
\pt\eta(\ueb) + \px q(\ueb) =&\eps \eta'(\ueb) \pxx\ueb +\beta^2\eta'(\ueb)\ptxxxx\ueb \\
=& I_{1,\,\eps,\,\beta}+I_{2,\,\eps,\,\beta}+ I_{3,\,\eps,\,\beta} + I_{4,\,\eps,\,\beta},
\end{align*}
where $I_{1,\,\eps,\,\beta},\,I_{2,\,\eps,\,\beta},\, I_{3,\,\eps,\,\beta},\, I_{4,\,\eps,\,\beta}$ are defined in \eqref{eq:12000}.

Fix $T>0$. Arguing as \cite[Lemma $3.2$]{Cd2}, we have that $I_{1,\,\eps,\,\beta}\to0$ in $H^{-1}((0,T) \times\R)$, and $\{I_{2,\,\eps,\,\beta}\}_{\eps,\beta >0}$ is bounded in $L^1((0,T)\times\R)$.\\
We claim that
\begin{equation*}
I_{3,\,\eps,\,\beta}\to0 \quad \text{in $H^{-1}((0,T) \times\R),\,T>0,$ as $\eps\to 0$.}
\end{equation*}
By \eqref{eq:beta-eps-2} and Lemma \ref{lm:n2},
\begin{align*}
&\norm{ \beta^2\eta'(\ueb)\ptxxx\ueb}^2_{L^2((0,T)\times\R)}\\
&\qquad\le \beta^4 \norm{\eta'}_{L^{\infty}(\R)}\norm{\ptxxx\ueb}^2_{L^2((0,T)\times\R)}\\
&\qquad= \norm{\eta'}_{L^{\infty}(\R)}\frac{\beta^4\eps}{\eps}\norm{\ptxxx\ueb}^2_{L^2((0,T)\times\R)}\\
&\qquad=\norm{\eta'}_{L^{\infty}(\R)}\frac{\beta\beta^3\eps}{\eps}\norm{\ptxxx\ueb}^2_{L^2((0,T)\times\R)}\le C_{0}\norm{\eta'}_{L^{\infty}(\R)}\eps^3\to0.
\end{align*}
Let us show that
\begin{equation*}
I_{4,\,\eps,\,\beta}\to0\quad \text{in $L^1((0,T) \times\R),\,T>0$.}
\end{equation*}
Thanks to \eqref{eq:beta-eps-2}, Lemmas \ref{lm:38}, \ref{lm:n2} and the H\"older inequality,
\begin{align*}
&\norm{\beta^2\eta''(\ueb)\px\ueb\ptxxx\ueb}_{L^1((0,T)\times\R)}\\
&\qquad\le\beta^2\norm{\eta''}_{L^{\infty}(\R)}\int_{0}^{T}\!\!\!\int_{\R}\vert\px\ueb\ptxxx\ueb\vert dsdx\\
&\qquad=\norm{\eta''}_{L^{\infty}(\R)}\frac{\beta^2\eps}{\eps}\norm{\px\ueb}_{L^2((0,T)\times\R)}\norm{\ptxxx\ueb}_{L^2((0,T)\times\R)}\\
&\qquad=\norm{\eta''}_{L^{\infty}(\R)}\frac{\beta^{\frac{1}{2}}\beta^{\frac{3}{2}}\eps}{\eps}\norm{\px\ueb}_{L^2((0,T)\times\R)}
\norm{\ptxxx\ueb}_{L^2((0,T)\times\R)}\\
&\qquad\le C_{0}\norm{\eta''}_{L^{\infty}(\R)}\eps\to 0.
\end{align*}
Arguing as \cite[Theorem $2.1$]{Cd5}, the proof is concluded.
\end{proof}

\appendix
\section{The Korteweg-de Vries equation: the first case}\label{appen1}
In this appendix, we consider the Korteweg-de Vries equation
\begin{equation}
\label{eq:kdv1}
\pt u +u\px u +\beta \pxxx u =0.
\end{equation}
We augment \eqref{eq:kdv1} with the initial condition
\begin{equation}
u(0,x)=u_{0}(x),
\end{equation}
on which we assume that
\begin{equation}
\label{eq:A1}
u_{0}\in L^{2}(\R), \quad -\infty <\int_{\R}u^3(x)dx <\infty.
\end{equation}
Observe that if $\beta\to 0$, we have \eqref{eq:BU}.

We study the dispersion-diffusion limit for \eqref{eq:kdv1}. Therefore, we fix two small numbers $\eps,\,\beta$ and consider the following third order approximation
\begin{equation}
\label{eq:A2}
\begin{cases}
\pt\ueb+ \ueb\px \ueb +\beta\pxxx\ueb =\eps\pxx\ueb, &\qquad t>0, \ x\in\R ,\\
\ueb(0,x)=u_{\eps,\beta,0}(x), &\qquad x\in\R,
\end{cases}
\end{equation}
where $u_{\eps,\beta,0}$ is a $C^\infty$ approximation of $u_{0}$ such that
\begin{equation}
\begin{split}
\label{eq:A3}
&u_{\eps,\,\beta,\,0} \to u_{0} \quad  \textrm{in $L^{p}_{loc}(\R)$, $1\le p < 2$, as $\eps,\,\beta \to 0$,}\\
&\norm{u_{\eps,\beta, 0}}^2_{L^2(\R)}+\beta \norm{\px u_{\eps,\beta,0}}^2_{L^2(\R)}\le C_0,\quad \eps,\beta >0,  \\
&-\infty <\int_{\R}u^3_{\eps,\,\beta,\,0}(x) dx <\infty , \quad \eps,\beta >0,
\end{split}
\end{equation}
and $C_0$ is a constant independent on $\eps$ and $\beta$.

The main result of this section is the following theorem.
\begin{theorem}
\label{th:main-A2}
Assume that \eqref{eq:A1} and \eqref{eq:A3} hold. Fix $T>0$,
if
\begin{equation}
\label{eq:A8}
\beta=\mathbf{\mathcal{O}}\left(\eps^{3}\right),
\end{equation}
then, there exist two sequences $\{\eps_{n}\}_{n\in\N}$, $\{\beta_{n}\}_{n\in\N}$, with $\eps_n, \beta_n \to 0$, and a limit function
\begin{equation*}
u\in L^{\infty}((0,T); L^2(\R)),
\end{equation*}
such that
\begin{itemize}
\item[$i)$] $u_{\eps_n, \beta_n}\to u$  strongly in $L^{p}_{loc}(\R^{+}\times\R)$, for each $1\le p <2$,
\item[$ii)$] $u$ a distributional solution of \eqref{eq:BU}.
\end{itemize}
Moreover, if
\begin{equation}
\label{eq:A30}
\beta=o\left(\eps^{3}\right),
\end{equation}
\begin{itemize}
\item[$iii)$] $u$ is the unique entropy solution of \eqref{eq:BU}.
\end{itemize}
\end{theorem}
Let us prove some a priori estimates on $\ueb$, denoting with $C_0$ the constants which depend only on the initial data.

Arguing as \cite{SC}, we have
\begin{equation}
\label{eq:A15}
\norm{\ueb(t,\cdot)}^2_{L^2(\R)}+2\eps\int_{0}^{t}\norm{\px\ueb(s,\cdot)}^2_{L^2(\R)}dx \le C_0,
\end{equation}
for every $t>0$.
\begin{lemma}\label{lm:A2}
Fix $T>0$. Assume that \eqref{eq:A8} holds. There exists $C_0>0$, independent on $\eps,\,\beta$ such that
\begin{equation}
\label{eq:A25}
\norm{\ueb}_{L^{\infty}((0,T)\times\R)}\le C_{0}\beta^{-\frac{1}{3}}.
\end{equation}
Moreover,
\begin{equation}
\label{eq:A30*}
\beta^{\frac{4}{3}} \norm{\px\ueb(t,\cdot)}^2_{L^2(\R)} + 2\beta^{\frac{4}{3}}\eps\int_{0}^{t}\norm{\pxx\ueb(s,\cdot)}^2_{L^2(\R)}ds\le C_0.
\end{equation}
\end{lemma}
\begin{proof}
Let $0<t<T$. Multiplying \eqref{eq:A2} by $-\ueb^2 -2\beta\pxx\ueb$, we have
\begin{equation}
\label{eq:A10}
\begin{split}
&\left(-\ueb^2 -2\beta\pxx\ueb\right)\pt\ueb+\left(-\ueb^2 -2\beta\pxx\ueb\right)\ueb\px\ueb\\
&\qquad\quad +\beta\left(- \ueb^2 - 2\beta\pxx\ueb\right)\pxxx\ueb = \eps\left(-\ueb^2 -2\beta\pxx\ueb\right)\pxx\ueb.
\end{split}
\end{equation}
Since
\begin{align*}
&\int_{\R}\left(-\ueb^2 -2\beta\pxx\ueb\right)\pt\ueb dx\\
&\qquad = -\frac{1}{3}\frac{d}{dt}\int_{\R}\ueb^3 dx+\beta\frac{d}{dt}\norm{\px\ueb(t,\cdot)}^2_{L^2(\R)},\\
&\int_{\R}\left(-\ueb^2 -2\beta\pxx\ueb\right)\ueb\px\ueb dx\\
&\qquad = -2\beta\int_{\R}\ueb\px\ueb\pxx\ueb dx,\\
&\beta\int_{\R}\left(- \ueb^2 - \beta\pxx\ueb\right)\pxxx\ueb dx\\
&\qquad = 2\beta\int_{\R}\ueb\px\ueb\pxx\ueb dx,\\
&\eps\int_{\R}\left(-\ueb^2 -2\beta\pxx\ueb\right)\pxx\ueb dx \\
&\qquad = 2\eps\int_{\R}\ueb(\px\ueb)^2 dx-2\beta\eps\norm{\pxx\ueb(t,\cdot)}^2_{L^2(\R)},
\end{align*}
integrating \eqref{eq:A10} on $\R$, we get
\begin{align*}
&\frac{d}{dt}\left(-\frac{1}{3}\int_{\R}\ueb^3 dx +\beta \norm{\px\ueb(t,\cdot)}^2_{L^2(\R)}\right)+ 2\beta\eps\norm{\pxx\ueb(t,\cdot)}^2_{L^2(\R)}\\
&\qquad= 2\eps\int_{\R}\ueb(\px\ueb)^2 dx \le 2\eps\norm{\ueb}_{L^{\infty}((0,T)\times\R)}\norm{\px\ueb(t,\cdot)}^2_{L^2(\R)}.
\end{align*}
\eqref{eq:A3}, \eqref{eq:A15} and an integration on $(0,t)$ give
\begin{align*}
&-\frac{1}{3}\int_{\R}\ueb^3 dx +\beta \norm{\px\ueb(t,\cdot)}^2_{L^2(\R)} + 2\beta\eps\int_{0}^{t}\norm{\pxx\ueb(s,\cdot)}^2_{L^2(\R)}ds\\
&\qquad \le C_{0}+2\eps \norm{\ueb}_{L^{\infty}((0,T)\times\R)}\int_{0}^{t}\norm{\px\ueb(s,\cdot)}^2_{L^2(\R)}ds\\
&\qquad \le C_{0} + C_0\norm{\ueb}_{L^{\infty}((0,T)\times\R)}.
\end{align*}
Again by \eqref{eq:A15}, we have
\begin{equation}
\begin{split}
\label{eq:A23}
&\beta \norm{\px\ueb(t,\cdot)}^2_{L^2(\R)} + 2\beta\eps\int_{0}^{t}\norm{\pxx\ueb(s,\cdot)}^2_{L^2(\R)}ds\\
&\qquad \le C_{0} + C_{0}\norm{\ueb}_{L^{\infty}((0,T)\times\R)} + \frac{1}{3}\int_{\R}\ueb^3 dx\\
&\qquad \le C_{0} +  C_{0}\norm{\ueb}_{L^{\infty}((0,T)\times\R)} + \frac{1}{3}\norm{\ueb}_{L^{\infty}((0,T)\times\R)}\norm{\ueb(t,\cdot)}^2_{L^2(\R)}\\
&\qquad \le C_{0}\left(1 +\norm{\ueb}_{L^{\infty}((0,T)\times\R)}\right).
\end{split}
\end{equation}
Due to \eqref{eq:A15}, \eqref{eq:A23} and the H\"older inequality,
\begin{align*}
\ueb^2(t,x)=&2\int_{-\infty}^{x}\ueb\px\ueb dy \le \int_{\R}\vert\ueb\vert\vert\px\ueb\vert dx \\
\le & \norm{\ueb(t,\cdot)}_{L^2(\R)}\norm{\px\ueb(t,\cdot)}_{L^2(\R)}\\
\le & \frac{C_{0}}{\sqrt{\beta}}\sqrt{\left(1 +\norm{\ueb}_{L^{\infty}((0,T)\times\R)}\right)},
\end{align*}
that is
\begin{equation*}
\norm{\ueb}_{L^{\infty}((0,T)\times\R)}^4 \le \frac{C_{0}}{\beta}\left(1 +\norm{\ueb}_{L^{\infty}((0,T)\times\R)}\right).
\end{equation*}
Arguing as \cite[Lemma $2.5$]{Cd2}, we have \eqref{eq:A25}.

Finally, \eqref{eq:A30*} follows from \eqref{eq:A25} and \eqref{eq:A23}.
\end{proof}

We begin by proving the following result.
\begin{lemma}\label{lm:9000}
Assume that  \eqref{eq:A1}, \eqref{eq:A3},  and \eqref{eq:A8} hold. Then, for any compactly
supported entropy-–entropy flux pair $(\eta, \,q)$, there exist two sequences $\{\eps_{n}\}_{n\in\N},\,\{\beta_{n}\}_{n\in\N}$, with $\eps_n,\,\beta_n\to0$, and a limit function
\begin{equation*}
u\in L^{\infty}((0,T);L^2(\R)),
\end{equation*}
such that
\eqref{eq:con-u-1} holds and
\begin{equation}
\label{eq:A40}
\textrm{$u$ is a distributional solution of \eqref{eq:BU}}.
\end{equation}
\end{lemma}
\begin{proof}
Let us consider a compactly supported entropy--entropy flux pair $(\eta, q)$. Multiplying \eqref{eq:A2} by $\eta'(\ueb)$, we have
\begin{align*}
\pt\eta(\ueb) + \px q(\ueb) =&\eps \eta'(\ueb) \pxx\ueb +\beta\eta'(\ueb)\pxxx\ueb \\
=& I_{1,\,\eps,\,\beta}+I_{2,\,\eps,\,\beta}+ I_{3,\,\eps,\,\beta} + I_{4,\,\eps,\,\beta},
\end{align*}
where
\begin{equation}
\begin{split}
\label{eq:1200013}
I_{1,\,\eps,\,\beta}&=\px(\eps\eta'(\ueb)\px\ueb),\\
I_{2,\,\eps,\,\beta}&= -\eps\eta''(\ueb)(\px\ueb)^2,\\
I_{3,\,\eps,\,\beta}&= \px(\beta\eta'(\ueb)\pxx\ueb),\\
I_{4,\,\eps,\,\beta}&= -\beta\eta''(\ueb)\px\ueb\pxx\ueb.
\end{split}
\end{equation}
Fix $T>0$. Arguing as \cite[Lemma $3.2$]{Cd2}, we have that $I_{1,\,\eps,\,\beta}\to0$ in $H^{-1}((0,T) \times\R)$, and $\{I_{2,\,\eps,\,\beta}\}_{\eps,\beta >0}$ is bounded in $L^1((0,T)\times\R)$.\\
We claim that
\begin{equation*}
I_{3,\,\eps,\,\beta}\to0 \quad \text{in $H^{-1}((0,T) \times\R),\,T>0,$ as $\eps\to 0$.}
\end{equation*}
By \eqref{eq:A8} and Lemma \ref{lm:A2},
\begin{align*}
&\norm{ \beta\eta'(\ueb)\pxx\ueb}^2_{L^2((0,T)\times\R)}\\
&\qquad\le \beta^2 \norm{\eta'}_{L^{\infty}(\R)}\norm{\pxx\ueb}^2_{L^2((0,T)\times\R)}\\
&\qquad= \norm{\eta'}_{L^{\infty}(\R)}\frac{\beta^2\eps}{\eps}\norm{\pxx\ueb}^2_{L^2((0,T)\times\R)}\\
&\qquad=\norm{\eta'}_{L^{\infty}(\R)}\frac{\beta^{\frac{2}{3}}\beta^{\frac{4}{3}}\eps}{\eps}\norm{\pxx\ueb}^2_{L^2((0,T)\times\R)}
\le C_{0}\norm{\eta'}_{L^{\infty}(\R)}\eps^2\to 0.
\end{align*}
Let us show that
\begin{equation*}
\{I_{4,\,\eps,\,\beta}\}\quad \text{is bounded in $L^1((0,T) \times\R),\,T>0$.}
\end{equation*}
Thanks to \eqref{eq:A8}, \eqref{eq:A15}, Lemma \ref{lm:A2}, and the H\"older inequality,
\begin{align*}
&\norm{\beta\eta''(\ueb)\px\ueb\pxx\ueb}_{L^1((0,T)\times\R)}\\
&\qquad\le\beta\norm{\eta''}_{L^{\infty}(\R)}\int_{0}^{T}\!\!\!\int_{\R}\vert\px\ueb\pxx\ueb\vert dsdx\\
&\qquad=\norm{\eta''}_{L^{\infty}(\R)}\frac{\beta^{\frac{1}{3}}\beta^{\frac{2}{3}}\eps}{\eps}\norm{\px\ueb}_{L^2((0,T)\times\R)}\norm{\pxx\ueb}_{L^2((0,T)\times\R)}\\
&\qquad\le C_{0}\norm{\eta''}_{L^{\infty}(\R)}\frac{\beta^{\frac{1}{3}}}{\eps}\le C_{0}\norm{\eta''}_{L^{\infty}(\R)}.
\end{align*}
Arguing as in \cite{SC}, we have \eqref{eq:A40}.
\end{proof}
\begin{lemma}\label{eq:10034}
Assume \eqref{eq:A1}, \eqref{eq:A3},  and \eqref{eq:A8} hold. Then, for any compactly
supported entropy–-entropy flux pair $(\eta, \,q)$, there exist two sequences $\{\eps_{n}\}_{n\in\N},\,\{\beta_{n}\}_{n\in\N}$, with $\eps_n,\,\beta_n\to0$, and a limit function
\begin{equation*}
u\in L^{\infty}((0,T);L^2(\R)),
\end{equation*}
such that
\eqref{eq:con-u-1} and \eqref{eq:u-entro-sol-12} hold.
\end{lemma}
\begin{proof}
Let us consider a compactly supported entropy--entropy flux pair $(\eta, q)$. Multiplying \eqref{eq:A2} by $\eta'(\ueb)$, we have
\begin{align*}
\pt\eta(\ueb) + \px q(\ueb) =&\eps \eta'(\ueb) \pxx\ueb +\beta\eta'(\ueb)\pxxx\ueb  \\
=& I_{1,\,\eps,\,\beta}+I_{2,\,\eps,\,\beta}+ I_{3,\,\eps,\,\beta} + I_{4,\,\eps,\,\beta},
\end{align*}
where $I_{1,\,\eps,\,\beta},\,I_{2,\,\eps,\,\beta},\, I_{3,\,\eps,\,\beta},\, I_{4,\,\eps,\,\beta}$ are defined in \eqref{eq:1200013}.

As in Lemma \ref{lm:259}, we have that $I_{1,\,\eps,\,\beta},\,I_{3,\,\eps,\,\beta}   \to 0$ in $H^{-1}((0,T)\times\R)$, $\{ I_{2,\,\eps,\,\beta}\}_{\eps,\beta>0}$ is bounded in $L^1((0,T)\times\R)$, while $I_{4,\,\eps,\,\beta}\to0$ in $L^1((0,T)\times\R)$.

Arguing as in \cite{LN}, we have \eqref{eq:u-entro-sol-12}.
\end{proof}

\begin{proof}[Proof of Theorem \ref{th:main-A2}]
Theorem \ref{th:main-A2} follows from Lemmas \ref{lm:9000} and  \ref{eq:10034}.
\end{proof}

\section{The Korteweg-de Vries equation: the second case.}\label{appen2}
In this appendix, we  argument  \eqref{eq:kdv1} with the following initial datum
\begin{equation}
\label{eq:N1}
u_{0}\in L^2(\R).
\end{equation}
We consider the approximation \eqref{eq:A2}, where $\ueb$ is a $C^{\infty}$ of $u_0$ such that
\begin{equation}
\begin{split}
\label{eq:N2}
&u_{\eps,\,\beta,\,0} \to u_{0} \quad  \textrm{in $L^{p}_{loc}(\R)$, $1\le p < 2$, as $\eps,\,\beta \to 0$,}\\
&\norm{u_{\eps,\beta, 0}}^2_{L^2(\R)}+\beta^{\frac{1}{2}}\norm{\px u_{\eps,\beta,0}}^2_{L^2(\R)}\le C_0,\quad \eps,\beta >0,
\end{split}
\end{equation}
and $C_0$ is a constant independent on $\eps$ and $\beta$.

The main result of this section is the following theorem.
\begin{theorem}
\label{th:main-N2}
Assume that \eqref{eq:N1} and  \eqref{eq:N2} hold. Fix $T>0$, if \eqref{eq:beta-eps-2} holds,
then, there exist two sequences $\{\eps_{n}\}_{n\in\N}$, $\{\beta_{n}\}_{n\in\N}$, with $\eps_n, \beta_n \to 0$, and a limit function
\begin{equation*}
u\in L^{\infty}((0,T); L^2(\R)),
\end{equation*}
such that
\begin{itemize}
\item[$i)$] $u_{\eps_n, \beta_n}\to u$  strongly in $L^{p}_{loc}(\R^{+}\times\R)$, for each $1\le p <2$,
\item[$ii)$] $u$ is the unique entropy solution of \eqref{eq:BU}.
\end{itemize}
\end{theorem}
Let us prove some a priori estimates on $\ueb$, denoting with $C_0$ the constants which depend only on the initial data
\begin{lemma}\label{lm:B3}
Fix $T>0$. Assume that \eqref{eq:beta-eps-2} holds. There exists $C_0>0$, independent on $\eps,\beta$ such that \eqref{eq:u-infty-3} holds.
Moreover,
\begin{equation}
\label{eq:N3}
\beta\norm{\px\ueb(t,\cdot)}^2_{L^2(\R)} + \frac{3\beta\eps}{2}\int_{0}^{t}\norm{\pxx\ueb(s,\cdot)}^2_{\R}ds\le C_0.
\end{equation}
\end{lemma}
\begin{proof}
Let $0<t<T$. Multiplying \eqref{eq:A2} by $-\beta^{\frac{1}{2}}\pxx\ueb$, an integration on $\R$ gives
\begin{equation}
\label{eq:N4}
\beta^{\frac{1}{2}}\frac{d}{dt}\norm{\px\ueb(t,\cdot)}^2_{L^2(\R)}+ 2\beta^{\frac{1}{2}}\eps\norm{\pxx\ueb(t,\cdot)}^2_{\R} = 2\beta^{\frac{1}{2}}\int_{\R}\ueb\px\ueb\pxx\ueb dx.
\end{equation}
Due to \eqref{eq:beta-eps-2} and the Young inequality,
\begin{equation}
\label{eq:N5}
\begin{split}
&2\beta^{\frac{1}{2}}\int_{\R}\vert\ueb\px\ueb\vert\vert\pxx\ueb\vert dx = 2\beta^{\frac{1}{2}}\int_{\R}\left\vert\frac{\ueb\px\ueb}{\eps^{\frac{1}{2}}}\right\vert\left\vert\eps^{\frac{1}{2}}\pxx\ueb\right\vert dx\\
&\qquad\le\frac{\beta^{\frac{1}{2}}}{\eps}\int_{\R}\ueb^2(\px\ueb)^2 dx +\frac{\beta^{\frac{1}{2}}\eps}{2}\norm{\pxx\ueb(t,\cdot)}^2_{L^2(\R)}\\
&\qquad\le C_{0}\eps\norm{\ueb}^2_{L^{\infty}((0,T)\times\R)}\norm{\px\ueb(t,\cdot)}^2_{L^2(\R)} +\frac{\beta^{\frac{1}{2}}\eps}{2}\norm{\pxx\ueb(t,\cdot)}^2_{L^2(\R)}.
\end{split}
\end{equation}
It follows from \eqref{eq:N4} and \eqref{eq:N5} that
\begin{align*}
\beta^{\frac{1}{2}}\frac{d}{dt}\norm{\px\ueb(t,\cdot)}^2_{L^2(\R)}+& \frac{3\beta^{\frac{1}{2}}\eps}{2}\norm{\pxx\ueb(t,\cdot)}^2_{L^2(\R)}\\
\le &C_{0}\eps\norm{\ueb}^2_{L^{\infty}((0,T)\times\R)}\norm{\px\ueb(t,\cdot)}^2_{L^2(\R)}.
\end{align*}
Integrating on $(0,t)$, from \eqref{eq:A15}  and \eqref{eq:N2}, we have
\begin{equation}
\label{eq:N6}
\begin{split}
\beta^{\frac{1}{2}}\norm{\px\ueb(t,\cdot)}^2_{L^2(\R)}&+ \frac{3\beta^{\frac{1}{2}}\eps}{2}\int_{0}^{t}\norm{\pxx\ueb(s,\cdot)}^2_{\R}ds\\
\le & C_{0} + C_{0}\eps\norm{\ueb}^2_{L^{\infty}((0,T)\times\R)}\int_{0}^{t}\norm{\px\ueb(t,\cdot)}^2_{L^2(\R)}ds\\
\le& C_{0}\left(1+\norm{\ueb}^2_{L^{\infty}((0,T)\times\R)}\right).
\end{split}
\end{equation}
Arguing as Lemma \ref{lm:50}, we have \eqref{eq:u-infty-3}.

\eqref{eq:N3} follows from \eqref{eq:u-infty-3} and \eqref{eq:N6}.
\end{proof}

We are ready for the proof of Theorem \ref{th:main-N2}.
\begin{proof}[Proof of Theorem \ref{th:main-N2}]
Let us consider a compactly supported entropy--entropy flux pair $(\eta, q)$. Multiplying \eqref{eq:A2} by $\eta'(\ueb)$, we have
\begin{align*}
\pt\eta(\ueb) + \px q(\ueb) =&\eps \eta'(\ueb) \pxx\ueb +\beta\eta'(\ueb)\pxxx\ueb  \\
=& I_{1,\,\eps,\,\beta}+I_{2,\,\eps,\,\beta}+ I_{3,\,\eps,\,\beta} + I_{4,\,\eps,\,\beta},
\end{align*}
where $I_{1,\,\eps,\,\beta},\,I_{2,\,\eps,\,\beta},\, I_{3,\,\eps,\,\beta},\, I_{4,\,\eps,\,\beta}$ are defined in \eqref{eq:1200013}.

As in Lemma \ref{lm:259}, we have that $I_{1,\,\eps,\,\beta},\,I_{3,\,\eps,\,\beta}   \to 0$ in $H^{-1}((0,T)\times\R)$, $\{ I_{2,\,\eps,\,\beta}\}_{\eps,\beta>0}$ is bounded in $L^1((0,T)\times\R)$.

We claim that
\begin{equation*}
I_{3,\,\eps,\,\beta}\to0 \quad \text{in $H^{-1}((0,T) \times\R),\,T>0,$ as $\eps\to 0$.}
\end{equation*}
By \eqref{eq:beta-eps-2} and \eqref{eq:N3},
\begin{align*}
&\norm{ \beta\eta'(\ueb)\pxx\ueb}^2_{L^2((0,T)\times\R)}\\
&\qquad\le \beta^2 \norm{\eta'}_{L^{\infty}(\R)}\norm{\pxx\ueb}^2_{L^2((0,T)\times\R)}\\
&\qquad= \norm{\eta'}_{L^{\infty}(\R)}\frac{\beta^2\eps}{\eps}\norm{\pxx\ueb}^2_{L^2((0,T)\times\R)}\\
&\qquad\le C_{0}\norm{\eta'}_{L^{\infty}(\R)}\eps^3\to 0.
\end{align*}
Let us show that
\begin{equation*}
I_{4,\,\eps,\,\beta}\to0 \quad \text{in $L^1((0,T) \times\R),\,T>0$.}
\end{equation*}
Thanks to \eqref{eq:beta-eps-2}, \eqref{eq:A15}, \eqref{eq:N3}, and the H\"older inequality,
\begin{align*}
&\norm{\beta\eta''(\ueb)\px\ueb\pxx\ueb}_{L^1((0,T)\times\R)}\\
&\qquad\le\beta\norm{\eta''}_{L^{\infty}(\R)}\int_{0}^{T}\!\!\!\int_{\R}\vert\px\ueb\pxx\ueb\vert dsdx\\
&\qquad=\norm{\eta''}_{L^{\infty}(\R)}\frac{\beta^{\frac{1}{2}}\beta^{\frac{1}{2}}\eps}{\eps}\norm{\px\ueb}_{L^2((0,T)\times\R)}\norm{\pxx\ueb}_{L^2((0,T)\times\R)}\\
&\qquad\le C_{0}\norm{\eta''}_{L^{\infty}(\R)}\eps\to0.
\end{align*}
Arguing as in \cite{LN}, the proof is concluded.
\end{proof}

\end{document}